\documentclass[a4paper,twoside,fleqn,notitlepage]{article}%
\usepackage{amsmath}
\usepackage{amsfonts}
\usepackage{amssymb}
\usepackage{graphicx}%
\providecommand{\U}[1]{\protect\rule{.1in}{.1in}}
\newtheorem{theorem}{Theorem}

\newtheorem{definition}[theorem]{Definition}
\newtheorem{example}[theorem]{Example}

\newtheorem{lemma}[theorem]{Lemma}

\newtheorem{proposition}[theorem]{Proposition}

\newenvironment{proof}[1][Proof]{\noindent\textbf{#1.} }{\ \rule{0.5em}{0.5em}}
\begin{document}

\title{Approximation properties of $q$-Bernoulli polynomials}
\author{M. Momenzadeh and I. Y. Kakangi\\Near East University\\Lefkosa, TRNC, Mersiin 10, Turkey \\Email: mohammad.momenzadeh@neu.edu.tr\\\ \ \ \ \ \ \ \ \ yusufkakangi@gmail.com}
\maketitle

\begin{abstract}
We study the $q-$analogue of Euler-Maclaurin formula and by introducing a new
$q$-operator we drive to this form. Moreover, approximation properties of
$q$-Bernoulli polynomials is discussed. We estimate the suitable functions as
a combination of truncated series of $q$-Bernoulli polynomials and the error
is calculated. This paper can be helpful in two different branches, first
solving differential equations by estimating functions and second we may apply
these techniques for operator theory.

\end{abstract}

\section{Introduction}

The \ present study has sough to investigate the approximation of suitable
function $f(x)$ as a linear combination of $q-$Bernoulli polynomials. This
study by using $q-$operators through a parallel way has achived to the kind of
Euler expansion for $f(x)$, and expanded the function in terms of
$q-$Bernoulli polynomials. This expansion offers a proper tool to solve
$q$-difference equations or normal differential equation as well. There are
many approaches to approximate the capable functions. According to properties
of $q$-functions many $q$-function have been used in order to approximate a
suitable function.For example, at \cite{haha}, some identities and formulae
for the q-Bernstein basis function, including the partition of unity property,
formulae for representing the monomials were studied. In addition, a kind of
approximation of a function in terms of Bernoulli polynomials are used in
several approaches for solving differential equations, such as [7, 20, 21].
This paper gives conditions to approximate capable functions as a linear
combination of $q$-Bernoulli polynomials as well as related examples. Also,
introduces the new $q-$operator to reach the $q-$analogue of Euler-Maclaurin formula.

This study, first, introduces some $q$-calculus concepts. There are several
types of $q-$Bernoulli polynomials and numbers that can be generated by
different $q-$exponential functions. Carlitz \cite{carlitz} is pionier of
introducing $q$-analogue of the Bernoulli numbers, he applied a sequence
$\left\{  \beta_{m}\right\}  _{m\geq0}$ in the middle of the 20$^{th}$
century;%
\begin{equation}
\sum_{k=0}^{m}\left(
\begin{array}
[c]{c}%
m\\
k
\end{array}
\right)  \beta_{k}q^{k+1}-\beta_{m}=\left\{
\begin{tabular}
[c]{ll}%
$1,$ & $m=1,$\\
$0,$ & $m>1.$%
\end{tabular}
\ \ \ \ \ \ \ \right.  \label{2}%
\end{equation}

First, this study makes an assumption that $\left\vert q\right\vert <1,$ and
this assumption is going to apply in the rest of the paper as well. If  $q$
tends to one from a left side the ordinary form would be reached.

Since Carlitz, there have been many distinct $q$-analogue of Bernoulli numbers
arising from varying motivations. In \cite{nazim}, they used improved
$q-$exponential functions to introduce a new class of $q-$Bernoulli
polynomials. In addition, they investigate some properties of these
$q$-Bernoulli polynomials.

Let us introduce a class of $q-$Bernoulli polynomials in a form of generating
function as follows;%

\begin{equation}
\frac{ze_{q}(zt)}{e_{q}(z)-1}=\sum_{n=0}^{\infty}\mathit{\beta}_{n,q}%
(t)\frac{z^{n}}{\left[  n\right]  _{q}!},\text{ \ \ \ \ \ \ \ \ \ }\left\vert
z\right\vert <2\pi. \label{55}%
\end{equation}

where $[n]_{q}$ is $q-$number and $q$-numbers factorial is defined by;%
\[
\left[  a\right]  _{q}=\frac{1-q^{a}}{1-q}\ \ \ \left(  q\neq1\right)
;\ \ \ \left[  0\right]  _{q}!=1;\ \ \ \ \left[  n\right]  _{q}!=\left[
n\right]  _{q}\left[  n-1\right]  _{q}\ \ \ \ n\in%
\mathbb{N}
,\ \ a\in%
\mathbb{C}
.
\]

The $q$-shifted factorial and $q$-polynomial coefficient are defined by the
following expressions respectively;
\begin{align*}
\left(  a;q\right)  _{0} &  =1,\ \ \ \left(  a;q\right)  _{n}=%
{\displaystyle\prod\limits_{j=0}^{n-1}}
\left(  1-q^{j}a\right)  ,\ \ \ n\in%
\mathbb{N}
,\\
\left(  a;q\right)  _{\infty} &  =%
{\displaystyle\prod\limits_{j=0}^{\infty}}
\left(  1-q^{j}a\right)  ,\ \ \ \ \left\vert q\right\vert <1,\ \ a\in%
\mathbb{C}
.
\end{align*}%
\[
\left(
\begin{array}
[c]{c}%
n\\
k
\end{array}
\right)  _{q}=\frac{\left(  q;q\right)  _{n}}{\left(  q;q\right)
_{n-k}\left(  q;q\right)  _{k}},
\]
$q$-standard terminology and notation can be found at \cite{kac} and
\cite{ernst}. We call $\mathit{\beta}_{n,q}$ Bernoulli number and
$\mathit{\beta}_{n,q}=\mathit{\beta}_{n,q}(0)$. In addition, $q$-analogue of
$\left(  x-a\right)  ^{n}$ is defined as;

\bigskip%
\[
\left(  x-a\right)  _{q}^{n}=\left\{
\begin{tabular}
[c]{ll}%
$1,$ & $n=0,$\\
$(x-a)(x-aq)...(x-aq^{n-1}),$ & $n\neq0.$%
\end{tabular}
\ \ \ \ \ \right.  ,
\]

In the standard approach to the $q-$calculus, two exponential function are
used.These $q-$exponentials are defined by;%

\begin{align*}
e_{q}\left(  z\right)   &  =\sum_{n=0}^{\infty}\frac{z^{n}}{\left[  n\right]
_{q}!}=\prod_{k=0}^{\infty}\frac{1}{\left(  1-\left(  1-q\right)
q^{k}z\right)  },\ \ \ 0<\left\vert q\right\vert <1,\ \left\vert z\right\vert
<\frac{1}{\left\vert 1-q\right\vert },\ \ \ \\
\ \ E_{q}(z)  &  =e_{1/q}\left(  z\right)  =\sum_{n=0}^{\infty}\frac
{q^{\frac{1}{2}n\left(  n-1\right)  }z^{n}}{\left[  n\right]  _{q}!}%
=\prod_{k=0}^{\infty}\left(  1+\left(  1-q\right)  q^{k}z\right)
,\ \ \ \ \ \ \ 0<\left\vert q\right\vert <1,\ z\in%
\mathbb{C}
,
\end{align*}

\bigskip$q-$shifted factorial can be expressed by Heine's binomial formula as follow;%

\[
\left(  a;q\right)  _{n}=\sum_{k=0}^{n}\left(
\begin{array}
[c]{c}%
n\\
k
\end{array}
\right)  _{q}q^{\frac{k(k-1)}{2}}\left(  -1\right)  ^{k}a^{k}%
\]

Let for some $0\leq\alpha<1,$ the function $|f(x)x^{\alpha}|$ is bounded on
the interval $\left(  0,A\right]  $, then Jakson integral defines as
\cite{kac};%

\[
\int\mathit{f}(x)d_{q}x=(1-q)x\sum_{i=0}^{\infty}q^{i}f(q^{i}x)
\]

Above expression converges to a function $F(x)$ on $\left(  0,A\right]  ,$
which is a $q-$antiderivative of $f(x)$. Suppose $0<a<b$, the definite
$q-$integral is defined as;%

\begin{align*}
\int\limits_{0}^{b}\mathit{f}(x)d_{q}x  &  =(1-q)b\sum_{i=0}^{\infty}%
q^{i}f(q^{i}b)\\
\int\limits_{a}^{b}\mathit{f}(x)d_{q}x  &  =\int\limits_{0}^{b}\mathit{f}%
(x)d_{q}x-\int\limits_{0}^{a}\mathit{f}(x)d_{q}x
\end{align*}

We need to apply some properties of the $q$-Bernoulli polynomials to prepare
the approximation conditions. These properties are listed below as a lemma;

\begin{lemma}
Following statements holds true:%

\begin{align}
a)\text{ }D_{q}\left(  \mathit{\beta}_{n,q}(t)\right)   &  =\left[  n\right]
_{q}\mathit{\beta}_{n-1,q}(t),\label{33}\\
b)\mathit{\beta}_{n,q}(t)  &  =\sum_{k=0}^{n}\binom{n}{k}_{q}\mathit{\beta
}_{k,q}t^{n-k},\\
c)\sum_{k=0}^{n}\binom{n}{k}_{q}\frac{\mathit{\beta}_{k,q}}{\left[
n-k+1\right]  _{q}}  &  =\delta_{0,n}\text{ where }\delta\text{ is kronecker
delta}\\
d)\int\limits_{0}^{1}\mathit{\beta}_{n,q}(t)d_{q}t  &  =0\text{ \ \ \ }n\neq0.
\end{align}

\end{lemma}

\begin{proof}
\bigskip taking $q-$derivative from both side of $\left(  2\right)  $ to prove
part (a). If we apply Cauchy product for generating formula $(4)$, it leads to
(b). Put $x=0$ at part (b) and change boundary of summation to reach (c). If
we take Jackson integral directly from $\mathit{\beta}_{n,q}(t)$ with the aid
of part (c), we can reach (d).
\end{proof}

\bigskip A few numbers of these polynomials and related $q-$Bernoulli numbers
can be expressed as follows

$%
\begin{array}
[c]{cc}%
\mathit{\beta}_{0,q}=1 & \mathit{\beta}_{0,q}(t)=1\\
\mathit{\beta}_{1,q}=-\frac{1}{[2]_{q}} & \mathit{\beta}_{1,q}(t)=t-\frac
{1}{[2]_{q}}\\
\mathit{\beta}_{2,q}=\frac{q^{2}}{[3]_{q}!} & \mathit{\beta}_{2,q}%
(t)=t^{2}-t+\frac{q^{2}}{[3]_{q}!}\\
\mathit{\beta}_{3,q}=\text{ }\frac{q^{3}(1-q)}{[5]_{q}+[4]_{q}-1} &
\mathit{\beta}_{3,q}(t)=t^{3}-\frac{[3]_{q}}{[2]_{q}}t^{2}+\frac{q^{2}%
}{[2]_{q}}t+\frac{q^{3}(1-q)}{[5]_{q}+[4]_{q}-1}%
\end{array}
$

\begin{definition}
$q-$Bernoulli polynomials of two variables are defined as a generating
function as follow%

\begin{equation}
\frac{te_{q}(xt)e_{q}(yt)}{e_{q}(t)-1}=\sum_{n=0}^{\infty}\mathit{\beta}%
_{n,q}(x,y)\frac{t^{n}}{\left[  n\right]  _{q}!},\text{ \ \ \ \ \ \ \ \ \ }%
\left\vert t\right\vert <2\pi.
\end{equation}

A simple calculation of generating function leads us to%

\begin{equation}
te_{q}(xt)+\frac{t}{e_{q}(t)-1}e_{q}(tx)=\frac{te_{q}(xt)}{e_{q}(t)-1}%
e_{q}(t)=\sum_{n=0}^{\infty}\mathit{\beta}_{n,q}(x,1)\frac{t^{n}}{\left[
n\right]  _{q}!},\text{ \ \ \ \ \ \ \ \ \ }\left\vert t\right\vert <2\pi.
\end{equation}

The LHS can be written as%

\begin{equation}
\sum_{n=0}^{\infty}\frac{t^{n+1}x^{n}}{\left[  n\right]  _{q}!}+\sum
_{n=0}^{\infty}\mathit{\beta}_{n,q}(x)\frac{t^{n}}{\left[  n\right]  _{q}%
!}\text{ \ \ \ \ \ \ \ \ }%
\end{equation}

By comparing the n$^{th}$ coefficients, we drive to difference equations as follow%

\begin{equation}
\mathit{\beta}_{n,q}(x,1)-\mathit{\beta}_{n,q}(x)=\left[  n\right]
_{q}!x^{n-1}\text{ \ \ \ \ \ \ \ \ \ }n\geq1.
\end{equation}

If we put $x=0,$ then $\mathit{\beta}_{n,q}(0,1)-\mathit{\beta}_{n,q}(0)=$
$\delta_{n,1}.$
\end{definition}

\section{\bigskip q-analogue of Euler-Maclaurin formula}

This section introduces $q-$operator to find $q-$analogue of Euler-Maclaurin
formula. The $q-$analogue of Euler-Maclaurin formula has been studied in
\cite{holhol}. They applied $q$-integral by parts to reach their formula. We
can not apply that approach to approximation, because it was written in term
of $p(x)=\mathit{\beta}_{n,q}(x-\left[  x\right]  )$. Moreover, the errors has
not been studied and our $q-$operator which is totally new, leads us to the
more applicable function. That is why, this study situate $q-$Taylor theorem
first\cite{ernst}

\begin{theorem}
If the function $f(x)$ is capable of expansion as a convergent power series
and if $q\neq$ root of unity, then
\begin{equation}
f(x)=\sum_{n=0}^{\infty}\frac{\left(  x-a\right)  _{q}^{n}}{\left[  n\right]
_{q}!}\left(  D_{q}^{n}f\right)  (a)\
\end{equation}

where $D_{q}(f(x))=\frac{f(xq)-f(x)}{x(q-1)}$ is $q-$derivative of $f(x),$ for
$x\neq0,$ \ we can define it at $x=0$ as a normal derivative.
\end{theorem}

\begin{definition}
\bigskip We define $H_{q}$ and $D_{h}$ operators as follow%

\begin{align}
H_{q}^{n}(x)  &  =\frac{h}{1-q}\left(  \frac{h}{1-q}+x\right)  \left(
\frac{h}{1-q}+\left[  2\right]  _{q}x\right)  ...\left(  \frac{h}{1-q}+\left[
n-1\right]  _{q}x\right)  \left(  1-q\right)  ^{n}\\
&  =\left(  x+h\right)  ^{n}\left(  \frac{x}{x+h};n\right)  _{q}=\sum
_{k=0}^{n}\left(
\begin{array}
[c]{c}%
n\\
k
\end{array}
\right)  _{q}q^{\frac{k(k-1)}{2}}\left(  x+h\right)  ^{n-k}\left(  -x\right)
^{k}\ \\
D_{h}(F(x))  &  =\frac{F(x+h)-F(x)}{h}=f(x)
\end{align}

In this definition we assume that $n\in%
\mathbb{N}
$ and the functions and values are well-defined $\left(  h\neq0,q\neq1\right)
.$ The fist equality is hold because of Heine's binomial formula.
\end{definition}

For a long time, mathematician have worked on the area of operators and they
solved several types of differential equations by using shifted-operators. The
$H-$operators rules like a bridge between the ordinary expansions and
$q$-analogue of these expansions. We may rewrite several formulae of these
area to the form of $q-$calculus such as [23, 24, 25, 26, 27]. Actually we may
write $q-$expansion of these functions. In a letter, Bernoulli concerned the
importence of this expansion for Leibnitz by these words" Nothing is more
elegent than the agreement, which you have observed between the numerical
power of the binomial and differetial expansions, there is no doubt that
something is hidden there" \cite{davis}

\begin{theorem}
\bigskip(Fundemental Theorem of $h$-Calculus) If $F(x)$ is an $h$-derivative
of $f(x)$ and $b-a\in h%
\mathbb{Z}
$, we have\cite{kac}%

\begin{equation}
\int f(x)d_{h}x=F(b)-F(a)
\end{equation}

where we define $h$-integral as follows%

\begin{equation}
\int f(x)d_{h}x=\left\{
\begin{array}
[c]{c}%
h\left(  f(a)+f(a+h)+...+f(b-h)\right)  \ \ \ \ \text{if }a<b\\
0\text{
\ \ \ \ \ \ \ \ \ \ \ \ \ \ \ \ \ \ \ \ \ \ \ \ \ \ \ \ \ \ \ \ \ \ \ \ \ \ \ \ \ \ \ \ \ \ \ \ \ \ if
}a=b\\
-h\left(  f(b)+f(b+h)+...+f(a-h)\right)  \ \text{ \ if }a>b
\end{array}
\right.  \
\end{equation}

\end{theorem}

Now in the aid of $q-$Taylor expansion, we may write $F(x+h)$ as follows;%

\begin{equation}
F(x+h)=\sum_{j=0}^{\infty}\frac{\left(  (x+h)-x\right)  _{q}^{j}}{\left[
j\right]  _{q}!}\left(  D_{q}^{j}F\right)  (x)\ =\sum_{j=0}^{\infty}%
\frac{D_{q}^{j}H_{q}^{j}}{\left[  j\right]  _{q}!}\left(  F(x)\right)
=e_{q}\left(  D_{q}H_{q}\right)  \left(  F(x)\right)  \
\end{equation}

Here, $e_{q}$ is $q-$shifted operator and can be expressed as expansion of
$D_{q}$ and $H_{q}.$ We may assume that $D_{h}\left(  F(x)\right)  =f(x)$ and then%

\begin{equation}
f(x)=\frac{F(x+h)-F(x)}{h}=\frac{e_{q}\left(  D_{q}H_{q}\right)  -1}{H_{q}%
}F(x)
\end{equation}

\bigskip Since the Jackson integral is $q-$antiderivative, we have%

\begin{equation}
\frac{H_{q}D_{q}}{e_{q}\left(  D_{q}H_{q}\right)  -1}\int f(x)d_{q}x=F(x)
\end{equation}

And in the aid of (2), the left hand side of the equation can be expressed as
a $q$-Bernoulli numbers, therefore%

\begin{align}
F(x)  &  =\sum_{n=0}^{\infty}\mathit{\beta}_{n,q}\frac{\left(  H_{q}%
D_{q}\right)  ^{n}}{\left[  n\right]  _{q}!}\int f(x)d_{q}x\\
&  =\int f(x)d_{q}x-\frac{h}{\left[  2\right]  _{q}}f(x)+\sum_{n=1}^{\infty
}\frac{\mathit{\beta}_{n,q}}{\left[  n\right]  _{q}!}\left(  D_{q}%
^{n-1}f(x)\right)  H_{q}^{n}(x)\\
&  =\int f(x)d_{q}x-\frac{h}{\left[  2\right]  _{q}}f(x)+\sum_{n=1}^{\infty
}\left(  \sum_{k=0}^{n}\left(
\begin{array}
[c]{c}%
n\\
k
\end{array}
\right)  _{q}q^{\frac{k(k-1)}{2}}\left(  h+x\right)  ^{n-k}\left(  -x\right)
^{k}\right)  \frac{\left(  D_{q}^{n-1}f(x)\right)  \mathit{\beta}_{n,q}%
}{\left[  n\right]  _{q}!}%
\end{align}

Thus we can state the following $q-$analogue of Euler-Maclaurian formula

\begin{theorem}
If the function $f(x)$ is capable of expansion as a convergent power series
and $f(x)$ decrease so rapidly with $x$ such that all normal derivatives
approach zero as $x\rightarrow\infty$, then we can express the series of
function $f(x)$ as follow%

\begin{align}
\sum_{n=a}^{\infty}f(n)  &  =\int_{a}^{\infty}f(x)d_{q}x+\frac{1}{\left[
2\right]  _{q}}\left(  f(a)\right)  -\\
&  \lim_{b\rightarrow\infty}\sum_{n=1}^{\infty}\left(  \sum_{k=0}^{n}\left(
\begin{array}
[c]{c}%
n\\
k
\end{array}
\right)  _{q}q^{\frac{k(k-1)}{2}}\left(  \left(  1+b\right)  ^{n-k}\left(
-b\right)  ^{k}-\left(  1+a\right)  ^{n-k}\left(  -a\right)  ^{k}\right)
\right)  \frac{h^{n}\left(  D_{q}^{n-1}f(a)\right)  \mathit{\beta}_{n,q}%
}{\left[  n\right]  _{q}!}%
\end{align}

\end{theorem}

\begin{proof}
In the aid of theorem (5) and (18), If we suppose that $h=1$ and $b-a\in%
\mathbb{N}
,$ we have%

\begin{align}
\sum_{n=a}^{b-1}f(n)  &  =\int\limits_{a}^{b}f(x)d_{1}x=F(b)-F(a)\\
&  =\int_{a}^{b}f(x)d_{q}x-\frac{1}{\left[  2\right]  _{q}}\left(
f(b)-f(a)\right) \\
&  +\sum_{n=1}^{\infty}\left(  \sum_{k=0}^{n}\left(
\begin{array}
[c]{c}%
n\\
k
\end{array}
\right)  _{q}q^{\frac{k(k-1)}{2}}\left(  \left(  1+b\right)  ^{n-k}\left(
-b\right)  ^{k}-\left(  1+a\right)  ^{n-k}\left(  -a\right)  ^{k}\right)
\right)  \frac{\left(  D_{q}^{n-1}f(b)-D_{q}^{n-1}f(a)\right)  \mathit{\beta
}_{n,q}}{\left[  n\right]  _{q}!}%
\end{align}

Let $f(x)$ decrease so rapidly with $x$ then $D_{q}f(x)$ can be estimated by
$\frac{-f(x)}{x(q-1)},$ tend $x$ to infinity and use L'Hopital to reach
$\frac{-f^{\prime}(x)}{(q-1)}.$Now If normal derivatives of $f(x)$ is tend to
zero, then $q-$derivatives is also tending zero. The same discussion make
$D_{q}^{n-1}f(b)=0$ for big enough $b.$ Therefore%

\begin{align}
\sum_{n=a}^{\infty}f(n)  &  =\int_{a}^{\infty}f(x)d_{q}x+\frac{1}{\left[
2\right]  _{q}}\left(  f(a)\right)  -\\
&  \lim_{b\rightarrow\infty}\sum_{n=1}^{\infty}\left(  \sum_{k=0}^{n}\left(
\begin{array}
[c]{c}%
n\\
k
\end{array}
\right)  _{q}q^{\frac{k(k-1)}{2}}\left(  \left(  1+b\right)  ^{n-k}\left(
-b\right)  ^{k}-\left(  1+a\right)  ^{n-k}\left(  -a\right)  ^{k}\right)
\right)  \frac{h^{n}\left(  D_{q}^{n-1}f(a)\right)  \mathit{\beta}_{n,q}%
}{\left[  n\right]  _{q}!}%
\end{align}

\end{proof}

\begin{example}
Let $f(x)=x^{s}$ where $s$ is a positive integer, then we can apply this
function at (22) where $a=0$ and $b-1\in%
\mathbb{N}
$%

\begin{align}
\sum_{n=0}^{b-1}n^{s}  &  =\frac{b^{s+1}}{\left[  s+1\right]  _{q}}-\left(
\frac{b^{s}}{\left[  2\right]  _{q}}\right) \\
&  +\sum_{n=2}^{s+1}\left(  \sum_{k=0}^{n}\left(
\begin{array}
[c]{c}%
n\\
k
\end{array}
\right)  _{q}q^{\frac{k(k-1)}{2}}\left(  \left(  1+b\right)  ^{n-k}\left(
-b\right)  ^{k}\right)  \right)  \left(
\begin{array}
[c]{c}%
s\\
n-1
\end{array}
\right)  _{q}\frac{b^{s-n+1}\mathit{\beta}_{n,q}}{\left[  n\right]  _{q}}%
\end{align}

This relation shows the sum of power in the combination of $q-$Bernoulli
numbers. when $q\rightarrow1$ from the right side, we have an ordinary form of
this relation. For the another forms of sum of power, see \cite{chan}
\end{example}

\begin{example}
Let $f(x)=e_{q}^{-x}=\frac{1}{E_{q}^{x}}$ then this function decreases so
rapidly with $x$ such that all normal derivatives approach zero as
$x\rightarrow\infty.$ For comfirming this, let us mention that $E_{q}^{x},$
for some fixed $0<|q|<1$ and $|-x|<\frac{1}{|1-q|}$ is increasing rapidly,
since $\frac{d}{dx}\left(  E_{q}^{x}\right)  =\sum\limits_{j=1}^{\infty
}(1-q)q^{j}\prod\limits_{\substack{k=0\\k\neq j}}^{\infty}\left(
1+(1-q)q^{k}x\right)  >0.$%

\begin{equation}
\sum_{n=0}^{\infty}e_{q}^{-n}=1+\frac{1}{\left[  2\right]  _{q}}%
-\lim_{b\rightarrow\infty}\sum_{n=1}^{\infty}\left(  \sum_{k=0}^{n}\left(
\begin{array}
[c]{c}%
n\\
k
\end{array}
\right)  _{q}q^{\frac{k(k-1)}{2}}\left(  \left(  1+b\right)  ^{n-k}\left(
-b\right)  ^{k}\right)  \right)  \frac{\mathit{\beta}_{n,q}}{\left[  n\right]
_{q}!}%
\end{equation}

Moreover, this is $q$-analogue of $\sum_{n=0}^{\infty}e^{-n}=\frac{-1}%
{e^{-1}-1}=\frac{3}{2}+\sum_{n=1}^{\infty}\frac{\mathit{\beta}_{2n}}{\left(
2n\right)  !},$ where $\mathit{\beta}_{2n}$ is a normal Bernoulli numbers that
is generated by $\frac{-1}{e^{-1}-1}.$
\end{example}

\section{approximation by q-Bernoulli polynomial}

We know that the class of $q$-Bernoulli polynomials are not in a form of
orthogonal polynomials. In addition, we may apply Gram-Schmit algorithm to
make these polynomials orthonormal then, according to Theorem 8.11
\cite{rudin} we have the best approximation in the form of Fourier series.
Now, Instead of using that algorithm, we apply properties of lemma (1) to
achieve an approximation. Let $H=L_{q}^{2}\left[  0,1\right]  =\left\{
f_{q}:\left[  0,1\right]  \rightarrow\left[  0,1\right]  |\int\limits_{0}%
^{1}\left\vert f_{q}^{2}(t)\right\vert d_{q}t<\infty\right\}  $ be the Hilbert
Space\cite{mansour}, then $\mathit{\beta}_{n,q}(x)\in H$ for $n=0,...,N$ and
$Y=Span\left\{  \mathit{\beta}_{0,q}(x),\mathit{\beta}_{1,q}%
(x),...,\mathit{\beta}_{N,q}(x)\right\}  $ is finite dimensional vector
subspace of $H$. Unique best approximation for any arbitrary elements of $H$
like $h,$ is $\widehat{h}\in Y$ such that for any $y\in Y$ the inequality
$\left\Vert h-\widehat{h}\right\Vert _{2,q}\leq$ $\left\Vert h-y\right\Vert
_{2,q},$ where the norm is defined by $\left\Vert f\right\Vert _{2,q}:=\left(
\int\limits_{0}^{1}\left\vert f_{q}^{2}(t)\right\vert d_{q}t\right)
^{\frac{1}{2}}.$ Following proposition determine the coefficient of
$q-$Bernoulli polynomials, when we estimate any $f\in L_{q}^{2}\left[
0,1\right]  $ by truncated $q-$Bernoulli series. In fact, in order to see how
well a certain partial sum approximates the actual value, we would like to
drive a formula similar to (19), but with the infinite sum on the right hand
side replaced by the $N$th partial sum $S_{N\text{,}}$ plus an additional term
$R_{N.}$

\bigskip Sippose $a\in%
\mathbb{Z}
$ and $b=a+1.$Consider the N$^{th}$ partial sum%

\[
S_{N}=\sum_{k=0}^{N}\frac{\mathit{\beta}_{k,q}}{\left[  k\right]  _{q}%
!}\left(  H_{q}^{k}D_{q}^{k}\left(  f\right)  (a+1)-H_{q}^{k}D_{q}^{k}\left(
f\right)  (a)\right)
\]

Let $g(x,y)=\frac{\mathit{\beta}_{N,q}(x,y)}{\left[  N\right]  _{q}!},$ then
in the aid of lemma (1) we have $D_{q}^{n-k}g(x,y)=\frac{\mathit{\beta}%
_{k,q}(x,y)}{\left[  k\right]  _{q}!}$ and we have

\begin{align*}
S_{N}  &  =\sum_{k=0}^{N}\left(  \frac{\mathit{\beta}_{k,q}(0,0)}{\left[
k\right]  _{q}!}H_{q}^{k}D_{q}^{k}\left(  f\right)  (a+1)-\frac{\mathit{\beta
}_{k,q}(0,1)}{\left[  k\right]  _{q}!}H_{q}^{k}D_{q}^{k}\left(  f\right)
(a)\right)  -D_{q}\left(  f\right)  (a)\\
&  =\sum_{k=0}^{N}\left(  D_{q}^{N-k}g(0,0)H_{q}^{k}D_{q}^{k}\left(  f\right)
(a+1)-D_{q}^{N-k}g(0,1)H_{q}^{k}D_{q}^{k}\left(  f\right)  (a)\right)
-D_{q}\left(  f\right)  (a)\\
&  =D_{q}\left(  f\right)  (a)-\sum_{k=0}^{N}D_{q}^{N-k}g(0,x)H_{q}^{k}%
D_{q}^{k}\left(  f\right)  (a+1-x)|_{x=0}^{x=1}%
\end{align*}

By taking $q-$derivative of the summation, we may rewrite this partial sum by
integral presentation, Moreover this relation give a boundary for $R_{N}(x).$

\begin{proposition}
\bigskip Let $f\in L_{q}^{2}\left[  0,1\right]  $ and be estimated by
truncated $q-$Bernoulli series $\sum_{n=0}^{N}C_{n}\mathit{\beta}_{n,q}(x).$
Then $C_{n}$ coefficients for $n=0,1,...,N$ can be calculated as an integral
form $C_{n}=\frac{1}{\left[  n\right]  _{q}!}\int\limits_{0}^{1}D_{q}%
^{(n)}f(x)d_{q}x$. In addition we can write $f(x)$ in the following form%

\[
f(x)=\int_{0}^{1}f(x)d_{q}x-\frac{h}{\left[  2\right]  _{q}}f(x)+\sum
_{n=1}^{N}\left(  \frac{1}{\left[  n\right]  _{q}!}\int\limits_{0}^{1}%
D_{q}^{(n)}f(x)d_{q}x\right)  \mathit{\beta}_{n,q}(x)+R_{q,n}(x)
\]

Moreover $R_{q,n}(x)$ as a reminder part is bounded by $\frac{2^{n}}{\left[
N\right]  _{q}!}Sup_{x\in\left[  0,1\right]  }|\mathit{\beta}_{n,q}%
(x)|Sup_{x\in\left[  0,1\right]  }|D_{q}^{(n)}f(x)|$

\begin{proof}
In fact we approximate $f(x)$ as a linear combinations of $q-$Bernoulli
polynomials and we assume that $f(x)\simeq\sum_{n=0}^{N}C_{n}\mathit{\beta
}_{n,q}(x),$ so take the Jackson integral from both sides lead us to%

\[
\int_{0}^{1}f(x)d_{q}x\simeq\sum_{n=0}^{N}C_{n}\int_{0}^{1}\mathit{\beta
}_{n,q}(x)d_{q}x=C_{0}\int_{0}^{1}\mathit{\beta}_{0,q}(x)d_{q}x+C_{1}\int
_{0}^{1}\mathit{\beta}_{1,q}(x)d_{q}x+...+C_{N}\int_{0}^{1}\mathit{\beta
}_{N,q}(x)d_{q}x
\]

In the aid of Lemma 1 part (d) all the terms at the right side except the
first one has to be zero. we calculate $\mathit{\beta}_{0,q}(x)=1,$ so
$C_{0}=\frac{1}{\left[  0\right]  _{q}!}\int\limits_{0}^{1}f(x)d_{q}x.$Using
part (a) of that lemma for $q-$derivative of $q-$Bernoulli polynomial gather
by taking $q$-derivative of $f(x)$ leads to
\[
\int_{0}^{1}D_{q}\left(  f(x)\right)  d_{q}x\simeq\sum_{n=0}^{N}C_{n}\left[
n\right]  _{q}!\int_{0}^{1}\mathit{\beta}_{n-1,q}(x)d_{q}x=C_{1}\int_{0}%
^{1}\mathit{\beta}_{0,q}(x)d_{q}x=C_{1}%
\]

Repeating this procedure $n$-times, yields the form of $C_{n}$ as we mention
it at theorem. We mention that if $x\in\left[  0,1\right]  $ and $0<|q|<1$,
then $H_{q}^{n}(x)$ for $h=1$ is bounded by $2^{n}.$In addition, the reminder
part can be presented by integral forms and this boundary can be found easily.
\end{proof}
\end{proposition}

\subsection{Further works}

In this paper, we introduced a proper tool to approximate a given capable
function by combinations of q-Bernoulli polynomials. In spite of a lot of
investigations on q-Bernoulli polynomials, study of approximation properties
of the q-Bernoulli polynomials are not studied. We may apply these results to
solve q-difference equation, $q$-analogue of the things that is done in
\cite{tohid 1} or \cite{tohid 2}. The techniques of H-operator can be applied
in operator theory.

\section{\bigskip conflict of interest}

The authors declare that there is no conflict of interest regarding the
publication of this paper


\begin{thebibliography}{99}                                                                                               %


\bibitem {kac}kac, V. Cheung, P. (2000). \textit{Quantum Calculus.} USA: New
York. 

\bibitem {carlitz}Carlitz, L. (1948). Duke Math. Journal. $q$%
\textit{-Bernoulli numbers and polynomials}.  (15) pp. 987--1000.

\bibitem {nazim}Mahmudov, N, I. and Momenzadeh, M. (2014). Abstract and
Applied Analysis. \textit{On a Class of q-Bernoulli, q-Euler, and q-Genocchi
Polynomials}. (2014) 10 pages.

\bibitem {Van Assche}Van Assche, W. (2001). The Ramanujan Journal.
\textit{Little }$q-$\textit{Legendre polynomials and irrationality of certain
Lambert} (5). pp. 295-310.

\bibitem {improved}Cie%
\'{}%
sli%
\'{}%
nski, L. (2011). Applied Mathematics Letters. \textit{Improved q-exponential
and q-trigonometric functions}.(24). Issue 12. pp. 2110--2114.

\bibitem {ernst}Ernst, T. (2000). \textit{the history of q-calculus and a new
method}. USA:Uppsala.

\bibitem {diff eq}Tohidi, E. Klicman, A. (2013). Hindawi Journal of applied
analysis. \textit{A Collocation Method Based on the Bernoulli Operational
Matrix for Solving Nonlinear BVPs Which Arise from the Problems in Calculus of
Variatio. }(2013) pp. 1-10.

\bibitem {evans}Evans, D, J. Raslan, K, R. (2005). Int. J. Comput. Math.
\textit{The adomain decomposition method for solving differential equation}
(82) pp. 49-54.

\bibitem {Jackson}Jackson, F, H. (1908). Trans. Roy. Soc. Edinb. \textit{On
q-functions and a certain difference operator.} (46). pp. 64--72.

\bibitem {Carm}Carmichael, R, D. (1913). Am. J. Math. \textit{On the theory of
linear difference equations.} (2)35. pp. 163--182.

\bibitem {chen}Chen, B, Q. Chen, X, Z. Li,  S. (2010). Acta Math. Sin. (Engl.
Ser.). \textit{Properties on solutions of some q-difference equations.}
(10)26. pp. 1877--1886 .

\bibitem {dobro}Dobrogowska, A. Odzijewicz, A. (2006). J. Comput. Appl. Math.
\textit{Second order q-difference equations solvable by factorization method.}
(1)193. pp. 319--346 .

\bibitem {abel}Abdel-Gawad, H, I. Aldailami, A, A. (2010). Appl. Math. Model.
\textit{On q-dynamic equations modelling and complexity. }(3)34. pp. 697--709 .

\bibitem {algin}Algin, A. (2011). Int. J. Theor. Phys. \textit{A comparative
study on q-deformed fermion oscillators.} (5)50. pp. 1554--1568.

\bibitem {dobro2}Dobrogowska, A. Odzijewicz, A. (2007). J. Phys. A Math.
Theor. \textit{Solutions of the q-deformed Schr%
\"{}%
odinger equation for special potentials.} (9)40. pp. 2023--2036 .

\bibitem {mab}Mabrouk, H. (2006). Fract.Calc.Appl.Anal. \textit{q-heat
operatorand q-Poisson's operator.} (3)9. pp. 265--286 .

\bibitem {nem}A. Nemri, A. Fitouhi, (2010). Journal of Matematiche.
\textit{Polynomial expansions for solution of wave equation in quantum
calculus. }(1)65. pp. 73--82 .

\bibitem {mansour}Annaby, M, H. Mansour, Z. (2012). $q$\textit{-fractional
calculus and equations}, USA: New York.

\bibitem {rudin}Rudin, W. (1964). \textit{Principles of mathematical
analysis.}USA: New York.

\bibitem {tohid 1}Tohidi, E. zak, M. (2016). Mediterranean journal of
mathematics. \textit{A new matrix approach for solving second-order linear
matrix partial differential equations}. (13).  pp 1353-1376.

\bibitem {tohid 2}Zogheib, B. Tohidi, E. (2016). Applied mathematics and
computation. \textit{A new matrix method for solving two-dimentional time
dependent diffusion equation with Dirichlet boundary co\textit{n}ditions.
}(291).  pp 1-12.

\bibitem {chan}Chan, O. Manna, D. (2010). Ams clas. \textit{A new q-analogue
for Bernoulli numbers.} (11).68. pp. 1-6.

\bibitem {carmi}Carmichael, R. (1855). Longman, brown, green, and longmans.
\textit{A treatise on the calculus of operations}. pp 92-103.

\bibitem {ricci}Ricci, P. Tavkelidze, I. (2009).  Journal of mathematical
sciences. \textit{An introduction to opertional techniques and special
polynomials}.  (157)1. pp. 1-7.

\bibitem {ricci2}Ricci, P. Tavkelidze, I. (2009).  Journal of mathematical
sciences. \textit{Operational methods and solutions of boundary-value problems
with periodic data}. (157)1. pp. 1-8.

\bibitem {davis}Davis, H. (1936). \textit{The thory of linear operators}. USA: Principia.

\bibitem {bool}Boole, G. (1880). \textit{A tretise on the calculus of finite
differences.} USA: Philadelphia.

\bibitem {haha}Goldman, R. Simeonov, P. Simsek, Y. (2014). SIAM Journal of
Discrete Mathematics. \textit{Generating functions for the }$q-$%
\textit{Bernstein bases}.  (28). \ pp. 1009-1025.

\bibitem {holhol}Hegazi, A. Mansour, M. (2006). Journal of nonlinear
mathematical physics. \textit{A note on }$q$\textit{-Bernoulli numbers and
polynomials}. (13)1. pp. 9-18.
\end{thebibliography}
\end{document}